\documentclass[12pt]{article}

\usepackage[centertags]{amsmath}
\usepackage{amsfonts}
\usepackage{amssymb}
\usepackage{amsthm}
\usepackage{newlfont}

\newlength{\defbaselineskip}
\newcommand{\setlinespacing}[1]%
           {\setlength{\baselineskip}{#1 \defbaselineskip}}

\theoremstyle{plain}
\newtheorem{thm}{Theorem}[section]

\newtheorem{lem}[thm]{Lemma}

\theoremstyle{definition}
\newtheorem{defn}[thm]{Definition}

\numberwithin{equation}{section}

\begin{document}

\newcommand{\ol }{\overline}
\newcommand{\ul }{\underline }
\newcommand{\ra }{\rightarrow }
\newcommand{\lra }{\longrightarrow }
\newcommand{\ga }{\gamma }
\newcommand{\st }{\stackrel }
\newcommand{\scr }{\scriptsize }

\title{\Large\textbf{Some outer commutator multipliers and capability of nilpotent products of cyclic groups}}
\author{\textbf{ Mohsen Parvizi and Behrooz Mashayekhy}\\ Department of Pure Mathematics,\\ Center of Excellence in Analysis on Algebraic Structures,\\
          Ferdowsi University of Mashhad,
          Mashhad, Iran. \\
          Email: parvizi@um.ac.ir    and    bmashf@um.ac.ir}
\date{ }
\maketitle

\begin{abstract}
In this paper, first we obtain an explicit formula for an outer commutator multiplier of nilpotent products of cyclic groups  with respect to the variety $[\mathfrak{N}_{c_1},\mathfrak{N}_{c_2}]$, $\mathfrak{N}_{c}M(\mathbb{Z}\st{n}*
\mathbb{Z}\st{n}* \cdots \st{n}* \mathbb{Z}\st{n}*
\mathbb{Z}_{r_1}\st{n}* \mathbb{Z}_{r_2}\st{n}* \cdots \st{n}*
\mathbb{Z}_{r_t})$ where $r_{i+1}\mid r_i \ \ (1\leq i\leq t-1)$,  $c_1+c_2+1\geq n$, $2c_2-c_1>2n-2$ and $(p,r_1)=1$ for all prime less than or equal $c_1+c_2+n$, second we give a necessary condition for these groups to be $[\mathfrak{N}_{c_1},\mathfrak{N}_{c_2}]$-capable.
\end{abstract}
2010{\it Mathematics Subject Classification}: 20E34, 20F12.

\section{Introduction and preliminaries}
\label{sec1}
Schur \cite{Schur} in 1907 and Wiegold \cite{Wiegold} in 1971 obtained a
structure of the Schur multiplier of the direct product of two
finite groups as follows:
\begin{center}
$M(A\times B)\cong M(A)\oplus M(B)\oplus \frac{[A,B]}{[A,B,A*B]}$,
where $\frac{[A,B]}{[A,B,A*B]}\cong A_{ab}\otimes B_{ab}$.
\end{center}
In 1979, Moghaddam \cite{Moghaddam} and in 1998, Ellis \cite{Ellis},
succeeded to extend the above result to obtain a structure of the
$c$-nilpotent multiplier of the direct product of two groups, ${\cal
N}_cM(A\times B)$. Also in 1997 M. R. R. Moghaddam and the second author in a joint paper \cite{Mashayekhy8}
presented an explicit formula for the $c$-nilpotent multiplier
of a finite abelian group.

Tahara \cite{Tahara} and Haebich \cite{Haebich1}
concentrated on the Schur multiplier of semidirect products of groups.
Also the second author worked on the Baer invariants of a semidirect
product in \cite{Mashayekhy9,Mashayekhy12}.

In 1972 Haebich \cite{Haebich2} presented a formula for the Schur multiplier
of a regular product of a family of groups. It is known that the
regular product is a generalization of the nilpotent product and the
last one is a generalization of the direct product, so Haebich's
result is a vast generalization of the Schur's result. Also,
Moghaddam \cite{Moghaddam}, in 1979 gave a formula similar to Haebich's
formula for the Schur multiplier of a nilpotent product. Moreover,
in 1992, Gupta and Moghaddam \cite{Gupta} presented an
explicit formula for the $c$-nilpotent multiplier of the $n$th
nilpotent product $\mathbb{Z}_2\st{n}*\mathbb{Z}_2$.

In 2001, the second author \cite{Mashayekhy10} found a structure similar to
Haebich's type for the $c$-nilpotent multiplier of a nilpotent
product of a family of cyclic groups. The $c$-nilpotent multiplier
of a free product of some cyclic groups was studied by the second
author \cite{Mashayekhy11} in 2002.

Recently, the authors \cite{Mashayekhy14,Parvizi17} concentrated on the Baer invariant
with respect to the variety of polynilpotent groups. We presented an explicit structure for some polynilpotent
multipliers of the $n$th nilpotent product of some infinite cyclic
groups \cite{Parvizi17} and also found explicit structures for all polynilpotent
multipliers of finitely generated abelian groups \cite{Mashayekhy14}.

In \cite{Parvizi18} the authors succeeded to determine an explicit
structure of
$[\mathfrak{N}_{c_1},\mathfrak{N}_{c_2}]M(\mathbb{Z}\st{n}*
\mathbb{Z}\st{n}* \cdots \st{n}* \mathbb{Z})$, where $c_1+c_2+1\geq n$ and $2c_2-c_1>2n-2$.
Also the second author, Hokmabadi and Mohammadzadeh \cite{MHM}, succeeded to compute an explicit structure of $\mathfrak{N}_{c}M(\mathbb{Z}\st{n}*
\mathbb{Z}\st{n}* \cdots \st{n}* \mathbb{Z}\st{n}*
\mathbb{Z}_{r_1}\st{n}* \mathbb{Z}_{r_2}\st{n}* \cdots \st{n}*
\mathbb{Z}_{r_t})$ and $\mathfrak{N}_{c_1,\ldots,c_t}M(\mathbb{Z}\st{n}*
\mathbb{Z}\st{n}* \cdots \st{n}* \mathbb{Z}\st{n}*
\mathbb{Z}_{r_1}\st{n}* \mathbb{Z}_{r_2}\st{n}* \cdots \st{n}*
\mathbb{Z}_{r_t})$ under some conditions. Hokmabadi \cite{Hokmabadi} presented the necessary and sufficient conditions for $\mathbb{Z}\st{n}*
\mathbb{Z}\st{n}* \cdots \st{n}* \mathbb{Z}\st{n}*
\mathbb{Z}_{r_1}\st{n}* \mathbb{Z}_{r_2}\st{n}* \cdots \st{n}*
\mathbb{Z}_{r_t}$ to be $\mathfrak{N_c}$-capable, where $r_{i+1}\mid r_i \ \ (1\leq i\leq t-1)$ and $(p,r_1)=1$ for all prime less that or equal $n$. In this article we intend to extend the results of \cite{Parvizi18} and \cite{MHM} to obtain an explicit structure of
$[\mathfrak{N}_{c_1},\mathfrak{N}_{c_2}]M(\mathbb{Z}\st{n}*
\mathbb{Z}\st{n}* \cdots \st{n}* \mathbb{Z}\st{n}*
\mathbb{Z}_{r_1}\st{n}* \mathbb{Z}_{r_2}\st{n}* \cdots \st{n}*
\mathbb{Z}_{r_t})$, where $r_{i+1}\mid r_i \ \ (1\leq i\leq t-1)$, $c_1+c_2+1\geq n$, $2c_2-c_1>2n-2$ and $(p,r_1)=1$ for all prime less that or equal $c_1+c_2+n$. Furthermore a necessary condition on $[\mathfrak{N_{c_1}},\mathfrak{N_{c_2}}]$-capability of $\mathbb{Z}\st{n}*
\mathbb{Z}\st{n}* \cdots \st{n}* \mathbb{Z}\st{n}*
\mathbb{Z}_{r_1}\st{n}* \mathbb{Z}_{r_2}\st{n}* \cdots \st{n}*
\mathbb{Z}_{r_t}$ is given.

\begin{defn}
Let $G$ be any group with a free presentation $G\cong F/R$, then, after R. Baer \cite{Baer}, the \textit{Baer
invariant} of $G$ with respect to a variety of groups $\mathfrak
{V}$, denoted by $\mathfrak{V}M(G)$, is defined to be
$$\mathfrak{V}M(G)=\frac{R\cap V(F)}{[RV^*F]}\ ,$$
where $V$ is the set of words of the variety $\mathfrak{V}$, $V(F)$
is the verbal subgroup of $F$ with respect to $\mathfrak{V}$ and
$$[RV^*F]=<v(f_1,\ldots ,f_{i-1},f_ir,f_{i+1},\ldots,f_n)v(f_1,\ldots,f_i,
\ldots,f_n)^{-1}\mid $$
$$r\in R, 1\leq i\leq n, v\in V ,f_i\in F, n\in \textbf{N}>.$$
\end{defn}

In special case of the variety $\mathfrak{A}$ of abelian groups, the
Baer invariant of $G$ will be the well-known notion the
\textit{Schur multiplier}
$$\frac{R\cap F'}{[R,F]}.$$

If $\mathfrak{V}$ is the variety of nilpotent groups of class at
most $c\geq1$, $\mathfrak{N}_c$, then the Baer invariant of $G$ with
respect to $\mathfrak{N}_c$ which is called the
\textit{$c$-nilpotent multiplier} of $G$, will be
$$\mathfrak{N}_cM(G)=\frac{R\cap \gamma_{c+1}(F)}{[R,\ _cF]},$$
where $\gamma_{c+1}(F)$ is the $(c+1)$-st term of the lower central
series of $F$ and $[R,\ _1F]=[R,F], [R,\ _cF]=[[R,\
_{c-1}F],F]$, inductively.

\begin{lem}
(Hulse and Lennox 1976). \textit{If $u$ and $w$ are any
two words and $v=[u,w]$ and $K$ is a normal subgroup of a group $G$,
then}
$$ [Kv^*G]=[[Ku^*G],w(G)][u(G),[Kw^*G]].$$
\end{lem}

\begin{proof}
See \cite[Lemma 2.9]{Hulse}.
\end{proof}

Now, using the above lemma, let $\mathfrak{V}$ be the outer
commutator variety $[\mathfrak{N}_{c_1},\mathfrak{N}_{c_2}]$, then
the Baer invariant of a group $G$ with respect to $\mathfrak{V}$, is
as follows:
$$[\mathfrak{N}_{c_1},\mathfrak{N}_{c_2}]M(G)\cong \frac{R\cap
[\ga_{c_1+1}(F),\ga_{c_2+1}(F)]}{[R,\ _{c_1}F,\ga_{c_2+1}(F)][R,\
_{c_2}F,\ga_{c_1+1}(F)]}
 \ \ (\star).$$

\begin{defn}
\textit{Basic commutators} are defined in the
usual way. If $X$ is a fully ordered independent subset of a free
group, the basic
commutators on $X$ are defined inductively over their weight as follows:\\
$(i)$ All the members of $X$ are basic commutators of weight one on $X$;\\
$(ii)$ assuming that $n>1$ and that the basic commutators of
weight less than $n$ on $X$ have been defined and ordered;\\
$(iii)$ a commutator $[b,a]$ is a basic commutator of weight $n$ on
$X$ if  $wt(a)+wt(b)=n,\ a<b$, and if $b=[b_1,b_2]$, then $b_2\leq
a$. The ordering of basic commutators is then extended to include
those of weight $n$ in any way such that those of weight less than
$n$ precede those of weight $n$. The natural way to define the order
on basic commutators of the same weight is lexicographically,
$[b_1,a_1]<[b_2,a_2]$ if $b_1<b_2$ or if $b_1=b_2$ and $a_1<a_2$.
\end{defn}

The next two theorems are vital in our investigation.

\begin{thm}
(Hall \cite{Mhall}). Let $F=<x_1,x_2,\ldots ,x_d>$ be a free group, then
$$ \frac {\ga_n(F)}{\ga_{n+i}(F)} \ \ , \ \ \ \  1\leq i\leq n$$
is the free abelian group freely generated by the basic commutators
of weights $n,n+1,\ldots ,n+i-1$ on the letters $\{x_1,\ldots
,x_d\}.$
\end{thm}

\begin{thm}
(Witt Formula \cite{Mhall}).The number of basic commutators of weight $n$ on
$d$ generators is given by the following formula:
$$ \chi_n(d)=\frac {1}{n} \sum_{m|n}^{} \mu (m)d^{n/m},$$
where $\mu (m)$ is the M\"{o}bius function.
\end{thm}

\begin{defn}
Let $\mathfrak{V}$ be a variety of groups defined by a set of laws
$V$. Then the \textit{verbal product} of a family of groups
$\{G_i\}_{i\in I}$ associated with the variety $\mathfrak{V}$ is
defined to be
$$\mathfrak{V}\prod_{i\in I}G_i=\frac {\prod^*G_i}{V(G)\cap
[G_i]^*},$$ where $G=\prod^{*}_{i\in I}G_i$ is the free product of
the family $\{G_i\}_{i\in I}$ and $[G_i]^*=\\ <[G_i,G_j]|i,j\in
I,i\neq j>^G$ is the cartesian subgroup of the free product $G$.
\end{defn}

The verbal product is also known as \textit{varietal product} or
simply \textit{$\mathfrak{V}$-product}. If $\mathfrak{V}$ is the
variety of all groups, then the corresponding verbal product is the
free product; if $\mathfrak{V}=\mathfrak{A}$ is the variety of all
abelian groups, then the verbal product is the direct product. Also,
if $\mathfrak{V}=\mathfrak{N}_c$ is the variety of nilpotent groups
of class at most $c\geq 1$, then the verbal product is called the
$c$\textit{th nilpotent product} of the $G_i$'s.


\section{Main Results}
\label{sec2}

Let $ G_{i}= \langle x_{i}| x_{i}^{k_{i}} \rangle \cong \mathbb{Z}
_{k _{i}} $  be the cyclic group of order $ k _{i} $ if $ k _{i} \geq
1 $, $(m+1\leq i \leq m+t)$ and the infinite cyclic group if $ k _{i} = 0 $, $(1\leq i \leq m)$. Let $$ 1
\rightarrow R _{i}= \langle x_{i}^{k_{i}} \rangle \rightarrow F_{i}=
\langle x_{i} \rangle \rightarrow G_{i} \rightarrow 1 $$ be a free
presentation for $ G_{i} $, so the $ n $th nilpotent product of the
family $ \{ G_{i}\} _{i \in I } $ is defined as follows: $$ \prod
^{\stackrel{n}{*}} _{i \in I }G_{i} = \frac{ \prod ^{* } _{i \in I }
G_{i}}{ \gamma _{n+1}(\prod  ^{* } _{i \in I }G_{i}) \cap [G_{i}]
^{*} _{i \in I } }.$$ It is easy to show that the cartesian
subgroup is the kernel of the natural homomorphism from $ \prod  ^{* } _{i
\in I }G_{i} $ to the direct product $ \prod ^{\times } _{i \in I
}G_{i} $. Since $ G_{i}^{,}$s are cyclic, it is easy to see that $
\gamma _{n+1}(\prod ^{* } _{i \in I } G_{i}) \subseteq [G_{i}] ^{*}
$ and hence $ \prod ^{\stackrel{n}{*}} _{i \in I }G_{i} = \prod ^{*
} _{i \in I } G_{i} / \gamma _{n+1}(\prod ^{* } _{i \in I }G_{i}) .$
Also, it is routine to check that a free presentation for the $n$th
nilpotent product of $ \prod ^{\stackrel{n}{*}} _{i \in I }G_{i}$ is
as follows: $$ 1 \rightarrow R=S \gamma _{n+1}(F)\rightarrow F=
\prod ^{* } _{i \in I } F_{i} \rightarrow \prod  ^{\stackrel{n}{*}}
_{i \in I }G_{i} \rightarrow 1,$$ where $S=\langle x_{i}^{k_{i}}|i
\in I \rangle ^{F}$ and $F$ is the free group on the set $$X=\{x_1,\cdots,x_m,x_{m+1},\cdots,x_{m+t}\}.$$

Now let $G\cong \mathbb{Z}\st{n}* \mathbb{Z}\st{n}* \cdots \st{n}*
\mathbb{Z}\st{n}* \mathbb{Z}_{r_1}\st{n}* \mathbb{Z}_{r_2}\st{n}*
\cdots \st{n}* \mathbb{Z}_{r_t}$ be the $n$th nilpotent product of
some cyclic groups ( $m$ copies of $\mathbb{Z}$). We intend
to obtain the structure of some outer commutator multipliers of $G$
of the form
$$[\mathfrak{N}_{c_1},\mathfrak{N}_{c_2}]M(G).$$
Using $(\star)$ we have
$$[\mathfrak{N}_{c_1},\mathfrak{N}_{c_2}]M(G)\cong
\frac{\ga_{n+1}(F)\cap
 [\ga_{c_1+1}(F),\ga_{c_2+1}(F)]}{[R,\
_{c_1}F,\ga_{c_2+1}(F)][R,\ _{c_2}F,\ga_{c_1+1}(F)]}.$$ Now if $c_1+c_2+1\geq n$, then we have

$$[\mathfrak{N}_{c_1},\mathfrak{N}_{c_2}]M(G)\cong \frac{[\ga_{c_1+1}(F),\ga_{c_2+1}(F)]}{[R,\
_{c_1}F,\ga_{c_2+1}(F)][R,\ _{c_2}F,\ga_{c_1+1}(F)]}.$$
Now consider the isomorphism
$$[\mathfrak{N}_{c_1},\mathfrak{N}_{c_2}]M(G)\cong \frac{[\mathfrak{N}_{c_1},\mathfrak{N}_{c_2}]M(\mathbb{Z}\st{n}*
\mathbb{Z}\st{n}* \cdots \st{n}* \mathbb{Z})}{\frac{[R,\
_{c_1}F,\ga_{c_2+1}(F)][R,\
_{c_2}F,\ga_{c_1+1}(F)]}{[\gamma_{c_1+n+1}(F),\gamma_{c_1+1}(F)][\gamma_{c_2+n+1}(F),\gamma_{c_2+1}(F)]}}.$$
To determine the explicit structure of
$[\mathfrak{N}_{c_1},\mathfrak{N}_{c_2}]M(G)$ it is enough to
determine the structure of
$[\mathfrak{N}_{c_1},\mathfrak{N}_{c_2}]M(\mathbb{Z}\st{n}*
\mathbb{Z}\st{n}* \cdots \st{n}* \mathbb{Z})$ and $$\frac{[R,\
_{c_1}F,\ga_{c_2+1}(F)][R,\
_{c_2}F,\ga_{c_1+1}(F)]}{[\gamma_{c_1+n+1}(F),\gamma_{c_1+1}(F)][\gamma_{c_2+n+1}(F),\gamma_{c_2+1}(F)]}.$$
But as we state in the following the authors \cite[Theorem 2.8]{Parvizi18} gave
the explicit structure of the first group.

\begin{thm}
If $2c_2-c_1> 2n-2$ and $c_1\geq c_2$, then
$\mathfrak{V}M(\mathbb{Z}\st{n}* \mathbb{Z}\st{n}*\ldots
\st{n}*\mathbb{Z})$ is the free abelian group with the following
basis:
$$D=\{ \ a[\ga_{c_1+n+1}(F),\ga_{c_2+1}(F)][\ga_{c_2+n+1}(F),\ga_{c_1+1}(F)]\ \ | \ \
a\in A-C \ \},$$ where
\begin{center}
$A=\{[\beta,\alpha] \ | \ \beta $ and $ \alpha $ are basic
commutators on $X$ such that $ \beta>\alpha, $ $ \ c_1+1\leq
wt(\beta)\leq c_1+n, \   c_2+1\leq wt(\alpha)\leq c_2+n \};$\\ \ \  \\

$C=\{ \ [\beta,\alpha] \ | \ \beta$ and $\alpha$ are basic
commutators on $X$ such that $\beta>\alpha$, $c_2+n+1\leq wt(\beta)$
, $c_1+1\leq wt(\alpha)$ , $wt(\beta)+wt(\alpha)\leq 2n+c_1+c_2+1
\}$.
\end{center}
\end{thm}
Now to determine the explicit structure of the second group we need
the following lemma and some definitions and theorems from [13]
which are as follows.

\begin{lem}
With the above notations and assumptions we have
$$[R,\ _{c_1}F,\ga_{c_2+1}(F)]\cong [S,\ _{c_1}F,\ga_{c_2+1}(F)] \pmod{[\gamma_{c_1+n+1}(F),\gamma_{c_1+1}(F)]}$$
and
$$[R,\ _{c_2}F,\ga_{c_1+1}(F)]\cong [S,\ _{c_2}F,\ga_{c_1+1}(F)] \pmod {[\gamma_{c_2+n+1}(F),\gamma_{c_2+1}(F)]}$$
\end{lem}
\begin{proof}
It is easy to see that for any normal subgroups of a group such as
$A$, $B$, $C$ and $H$ if $A\cong B \pmod C$ then we have $[A,H]\cong
[B,H] \pmod{[C,H]}$; now it is enough to show that $[R,F]\cong [S,F]
\pmod{\gamma_{n+2}(F)}$ which is straightforward. Now in [13] it is proved that
$$\frac{[R,\ _{c_1}F]}{\gamma_{c_1+n+1}(F)}$$ is the free abelian group with the
basis $\bigcup_{i=m}^{m+t-1}C_i$ in which $C_{i}=\{b^{m_{j+1}} \ | \
b$ is a basic commutator on $X$ and $x_{k+j}$ appears in $b\}$. The
same argument does hold for $c_2$.
\end{proof}

\begin{lem}
If $(p,r_1)=1$ for all prime $p$ less than or equal to $l-1$, then
$[S,\ _{c+1-1}F]\gamma _{c+l}(F)/ \gamma _{c+l}(F)$ is the free
abelian group with a basis $D_{c},$ where $$D_{c}=\{ b^{r_j} |
 \ b\ is\ a\ basic\ commutator\ of\ weight\ c+i\ on\ $$ $$ \ \ \ \
 \ x_1,...,x_m,...,x_{m+j}\ s.t. \ x_{m+j}\ appears\ in\ b \ 1\leq j \leq t \ and \ 1\leq i \leq l-1 \}.$$
\end{lem}

\begin{thm}
With the above notations and assumptions we have
$$[R,\ _{c_1}F,\ga_{c_2+1}(F)]\cong \langle D_1 \rangle \pmod {[\gamma_{c_1+n+1}(F),\gamma_{c_1+1}(F)][\gamma_{c_2+n+1}(F),\gamma_{c_2+1}(F)]}$$
in which $D_1=\{[b,c] \ | \ b\in D_{c_1}$ and $c$ is a basic
commutator on $X$ with $c_2+1\leq wt(c)\leq c_2+n\}$, and
$$[R,\ _{c_2}F,\ga_{c_1+1}(F)]\cong \langle D_2 \rangle \pmod {[\gamma_{c_1+n+1}(F),\gamma_{c_1+1}(F)][\gamma_{c_2+n+1}(F),\gamma_{c_2+1}(F)]}$$
in which $D_2=\{[b,c] \ | \ b\in D_{c_2}$ and $c$ is a basic
commutator on $X$ with $c_1+1\leq wt(c)\leq c_1+n\}$.
\end{thm}

\begin{proof}
Let $[a,b]$ be a generator of $[R,\ _{c_1}F,\ga_{c_2+1}(F)]$ so
$a\in [R,\ _{c_1}F]$ and $b\in \gamma_{c_2+1}(F)$, using the above
lemma and Hall's theorem we have $a=\prod a_i \alpha$ and $b=\prod
b_j \beta$ in which $a_i\in D_{c_1}$, $\alpha\in
\gamma_{c_1+n+1}(F)$, $b_j$'s are basic commutators on $X$ of
weights $c_2+1,...,c_2+n$ and $\beta\in \gamma_{c_2+n+1}(F)$. Now
$[a,b]$ is a product of elements of the form $[a_i,b_j]^{f_{ij}}$,
$[a_i,\beta]^{g_i}$, $[\alpha,b_j]^{h_j}$ and $[\alpha,\beta]^k$ in
which $h_{ij}, g_i, h_j, k\in \gamma_{c_2+1}(F)$. It is enough to
show that $[a_i,b_j,f_{ij}], [a_i,\beta], [\alpha,b_j]$ and
$[\alpha,\beta]$ are all elements of \linebreak
$[\gamma_{c_1+n+1}(F),\gamma_{c_1+1}(F)]$  $[\gamma_{c_2+n+1}(F),\gamma_{c_2+1}(F)]$
and this can be done by a routine calculation.
\end{proof}

The immediate consequence of the latest theorem is that
$$\frac{[R,\ _{c_1}F,\ga_{c_2+1}(F)][R,\ _{c_2}F,\ga_{c_1+1}(F)]}{[\gamma_{c_1+n+1}(F),\gamma_{c_1+1}(F)][\gamma_{c_2+n+1}(F),\gamma_{c_2+1}(F)]}\cong \langle Y \rangle \ \ \ \ (1)$$
in which $Y=\{[b,c]^{m_{i+1}}H \ | \ b$ and $c$ are basic
commutators on $X$ with $c_1+1\leq wt(b)\leq c_1+n, c_2+1\leq
wt(c)\leq c_2+n$ and $x_{i+1}$ appears in $[b,c]\}$.

Now considering Theorem 2.1 it is easy to see that every element of
$Y$ is a power of an element of $A-C$, so $Y$ is linearly
independent and that the group $$\frac{[R,\
_{c_1}F,\ga_{c_2+1}(F)][R,\
_{c_2}F,\ga_{c_1+1}(F)]}{[\gamma_{c_1+n+1}(F),\gamma_{c_1+1}(F)][\gamma_{c_2+n+1}(F),\gamma_{c_2+1}(F)]}$$
is a free abelian group.

Now, we can state and prove the main result of this section.

\begin{thm}
Let $G\cong\mathbb{Z}\st{n}*
\mathbb{Z}\st{n}* \cdots \st{n}* \mathbb{Z}\st{n}*
\mathbb{Z}_{r_1}\st{n}* \mathbb{Z}_{r_2}\st{n}* \cdots \st{n}*
\mathbb{Z}_{r_t})$ be the $nth$ nilpotent product of cyclic groups ($m$ copies of $\mathbb{Z}$) where $r_{i+1}\mid r_i \ \ (1\leq i\leq t-1)$, $c_1+c_2+1\geq n$, $2c_2-c_1>2n-2$ and $(p,r_1)=1$ for all prime less that or equal $c_1+c_2+n$ then
$$[\mathfrak{N}_{c_1},\mathfrak{N}_{c_2}]M(\mathbb{Z}\st{n}* \mathbb{Z}\st{n}* \cdots \st{n}* \mathbb{Z}\st{n}* \mathbb{Z}_{r_1}\st{n}* \mathbb{Z}_{r_2}\st{n}* \cdots \st{n}* \mathbb{Z}_{r_t}\cong \mathbb{Z}^{(d_1)}\oplus \mathbb{Z}_{r_1}^{(d_2)}\oplus\cdots \oplus\mathbb{Z}_{r_t}^{(d_t)}$$ in which

$(i)$ if $c_2+n<c_1+1$, then\\
$$d_1=(\sum_{i=c_1+1}^{c_1+n}\chi_i(m))(\sum_{c_2+1}^{c_1}\chi_i(m))$$\\
and\\
$$d_k=(\sum_{i=c_1+1}^{c_1+n}\chi_i(m+k+1)-\sum_{i=c_1+1}^{c_1+n}\chi_i(m+k))(\sum_{i=c_2+1}^{c_2+n}\chi_i(m+k))+$$\\
$$(\sum_{i=c_1+1}^{c_1+n}\chi_i(m+k))(\sum_{i=c_2+1}^{c_2+n}\chi_i(m+k+1)-\sum_{i=c_2+1}^{c_2+n}\chi_i(m+k))$$\\

$(ii)$ if $c_2+n\geq c_1+1$, then\\
$$d_1=(\sum_{i=c_1+1}^{c_1+n}\chi_i(m))(\sum_{c_2+1}^{c_1}\chi_i(m))+\chi_2(\sum_{i=c_1+1}^{c_2+n}\chi_i(m))$$\\
and\\
$$d_k=(\sum_{i=c_1+1}^{c_1+n}\chi_i(m+k+1)-\sum_{i=c_1+1}^{c_1+n}\chi_i(m+k))(\sum_{i=c_2+1}^{c_2+n}\chi_i(m+k))+$$\\
$$(\sum_{i=c_1+1}^{c_1+n}\chi_i(m+k))(\sum_{i=c_2+1}^{c_2+n}\chi_i(m+k+1)-\sum_{i=c_2+1}^{c_2+n}\chi_i(m+k))+$$\\
$$\chi_2(\sum_{i=c_1+1}^{c_2+n}\chi_i(m+k+1)-\chi_2(\sum_{i=c_1+1}^{c_2+n}\chi_i(m+k)$$
\end{thm}

\begin{proof}
First note that in case $(i)$ if $c_2+n<c_1+1$, then $A-C=A$ and in
case $(ii)$ if $c_2+n\geq c_1+1$, then
$|A-C|=(\sum_{i=c_1+1}^{c_1+n}\chi_i(m))(\sum_{c_2+1}^{c_1}\chi_i(m))+\chi_2(\sum_{i=c_1+1}^{c_2+n}\chi_i(m))$
\cite[Corollary 2.5]{Parvizi18}. Now it is easy to see that the number of
elements of $A-C$ which are not in $Y$ is $d_0$ and the number of
the elements of $Y$ with power $r_k$ is $d_k$.
\end{proof}


\section{Capability}
\label{sec3}
It was Hall \cite{PHall} who initiated studying capable groups in order to classify $p$-groups. Several papers were devoted to capability and varietal generalization such as Burns and Ellis \cite{Burns}, Ellis \cite{Elliscapable}, Moghaddam and Kayvanfar \cite{MoghaddamKayvanfar}, and Magidin \cite{Magidin1,Magidin2,Magidin3,Magidin4}. Recently Hokmabadi \cite{Hokmabadi} classified varietal capable groups in the class of nilpotent products of cyclic groups with respect to the variety of polynilpotent groups under some conditions. In this section we intend to determine varietal capability of such groups with respect to the variety $[\mathfrak{N_{c_1}},\mathfrak{N_{c_2}}]$ with some conditions. To do this we need the following theorem and lemma.

\begin{thm}
Let $H\cong \mathbb{Z}\st{c+n}* \mathbb{Z}\st{c+n}* \cdots \st{c+n}* \mathbb{Z}\st{c+n}* \mathbb{Z}_{r_1}\st{c+n}* \mathbb{Z}_{r_2}\st{c+n}* \cdots \st{c+n}* \mathbb{Z}_{r_t}$ where $r_{i+1}\mid r_i \ \ (1\leq i\leq t-1)$ and all prime numbers smaller than $c+n$ are coprime to $r_1$, then

\[ Z_c(H)=\left \{ \begin{array}{ll}
      \langle \gamma_{n+1}(H), x_1^{r_2}\rangle & \ m=0, \\ \langle \gamma_{n+1}(H), y_1^{r_1}\rangle & \ m=1, \\ \gamma_{n+1}(H) & \
m\geq 2.
\end{array} \right.  \]
\end{thm}

\begin{proof}
See \cite{Hokmabadi}.
\end{proof}

\begin{lem}
Let $A_i=\langle a_i \mid a_i^{\alpha_i} \rangle$, $1\leq i\leq t$ and $G=A_1\st{n}* \cdots \st{n}* A_t$
with $n\geq 2$ and all prime number less
than or equal to $n$ are coprime to $\alpha_1$. If $u\in G$ be an outer commutator on $a_1,\cdots,a_n$ such that only $a_{i_1},\cdots,a_{i_k}$ do appear in $u$ then $u^N=1$ in which $N=(a_{i_1},\cdots,a_{i_k})$ (the greatest common divisor of $a_{i_1},\cdots,a_{i_k}$).
\end{lem}

\begin{proof}
See \cite{Struik}.
\end{proof}

Now let $G\cong \mathbb{Z}\st{c+1}* \mathbb{Z}\st{c+1}* \cdots \st{c+1}* \mathbb{Z}\st{c+1}* \mathbb{Z}_{r_1}\st{c+1}* \mathbb{Z}_{r_2}\st{c+1}* \cdots \st{c+1}* \mathbb{Z}_{r_t}$ in which $c=c_1+c_2+1$. Similar to Theorem 3.1 we can prove

\begin{thm}
Let $H\cong \mathbb{Z}\st{c+n}* \mathbb{Z}\st{c+n}* \cdots \st{c+n}* \mathbb{Z}\st{c+n}* \mathbb{Z}_{r_1}\st{c+n}* \mathbb{Z}_{r_2}\st{c+n}* \cdots \st{c+n}* \mathbb{Z}_{r_t}$ where $r_{i+1}\mid r_i \ \ (1\leq i\leq t-1)$ and all prime numbers smaller than $c+n$ are coprime to $r_1$, then

\[ V^{\star}(H)=\left \{ \begin{array}{ll}
      \langle \gamma_{n+1}(H), x_1^{r_2}\rangle & \ m=0, \\ \langle \gamma_{n+1}(H), y_1^{r_1}\rangle & \ m=1, \ \ \ \ \ \ \ \ (1) \\ \gamma_{n+1}(H) & \
m\geq 2.
\end{array} \right.  \]
\end{thm}

\begin{proof}
Under the assumption of the theorem we have $V^{\star}(H)\subseteq Z_c(H)$. It is enough to show that $V^{\star}(H)$ contains the right hand side of (1). Clearly $\gamma_{n+1}(H)\subseteq V^{\star}(H)$. We show that $x_1^{r_2}\in V^{\star}(H)$. Let $h_i, 1\leq i\leq c_1$ and $h'_j, 1\leq j\leq c_2+1$ be arbitrary elements of the set $\{ x_1,x_2,\cdots,x_t\}$ we have $[x_1^{r_2},h_1,\ldots,h_{c_2},[h'_1,\ldots,h'_{c_2+1}]]=[x_1,h_1,\ldots,h_{c_2},[h'_1,\ldots,h'_{c_2+1}]]^{r_2}E_1^{f_1(r_2)}\ldots E_k^{f_k(r_2}$ in which $E_i$'s are basic commutators on $\{ x_1,x_2,\cdots,x_t\}$ and $w(E_1),\cdots,w(E_k)\geq c+1$ and $f_i(r_2)=\beta_1 {r_2 \choose 1}+\cdots+\beta_{\omega_i} {r_2 \choose \omega_i}$ where $\omega_i=w(E_i)-c\leq n$. So we have $r_2\mid f_i(r_2)$ and now by Lemma 3.2 \linebreak $[x_1^{r_2},h_1,\cdots,h_{c_2},[h'_1,\cdots,h'_{c_2+1}]]=1$ hence $x_1^{r_1}\in V^{\star}(H)$. The cases for which $m=1$ or $m\geq 2$ have a similar proof.
\end{proof}

The following Theorem is the main result of this section.

\begin{thm}
Let $G\cong \mathbb{Z}\st{n}* \mathbb{Z}\st{n}* \cdots \st{n}* \mathbb{Z}\st{n}* \mathbb{Z}_{r_1}\st{n}* \mathbb{Z}_{r_2}\st{n}* \cdots \st{n}* \mathbb{Z}_{r_t}$ where $r_{i+1}\mid r_i$ ($1\leq i \leq t-1$) and all prime numbers smaller than $c+n$ are coprime to $r_1$ then $G$ is $[\mathfrak{N_{c_1}},\mathfrak{N_{c_2}}]$-capable if $m\geq 2$ or $m=0$ and $r_1=r_2$.
\end{thm}

\begin{proof}
Let $H\cong \mathbb{Z}\st{c+n}* \mathbb{Z}\st{c+n}* \cdots \st{c+n}* \mathbb{Z}\st{c+n}* \mathbb{Z}_{r_1}\st{c+n}* \mathbb{Z}_{r_2}\st{c+n}* \cdots \st{c+n}* \mathbb{Z}_{r_t}$ then under the assumption of theorem $V^{\star}(H)=\gamma_{n+1}(H)$ so $H/V^{\star}(H)\cong H/\gamma_{n+1}(H)\cong G$.
\end{proof}

\section{Acknowledgement}
This research was supported by a grant from Ferdowsi University of Mashhad; (No. MP89203PAR)


\end{document}